\documentclass[11pt,reqno]{amsart}

\author{Peter Constantin}
\address{Department of Mathematics, Princeton University, Princeton, NJ 08544}
\email{const@math.princeton.edu}

\author{Mihaela Ignatova}
\address{Department of Mathematics, Temple University, Philadelphia, PA 19122}
\email{ignatova@temple.edu}

\author{Fizay-Noah Lee}
\address{Program in Applied and Computational Mathematics, Princeton University, Princeton, NJ 08544}
\email{fizaynoah@princeton.edu}

\usepackage{amsmath}
\usepackage{amssymb}
\usepackage{amsthm}
\usepackage{amsrefs}
\usepackage{dsfont}
\usepackage{mathrsfs}
\usepackage{stmaryrd}
\usepackage[all]{xy}
\usepackage[mathcal]{eucal}
\usepackage{verbatim}  
\usepackage{fullpage}  
\usepackage{cite}
\usepackage[margin=1in]{geometry}
\usepackage{amsmath, amsthm, amssymb}
\usepackage{times}
\usepackage{color}
\usepackage{hyperref}

\usepackage{chngcntr}
\counterwithout{equation}{section}
\counterwithout{equation}{subsection}
\counterwithout{equation}{subsubsection}

\newcommand{\pa}{\partial}
\newcommand{\la}{\label}
\newcommand{\fr}{\frac}
\newcommand{\na}{\nabla}
\newcommand{\be}{\begin{equation}}
\newcommand{\ee}{\end{equation}}
\newcommand{\ba}{\begin{array}{l}}
\newcommand{\ea}{\end{array}}

\newcommand{\beg}{\begin}

\renewcommand{\div}{{\mbox{div}\,}}
\newcommand{\D}{\Delta}

\renewcommand{\phi}{\varphi}

\providecommand{\norm}[1]{\lVert #1 \rVert}

\providecommand{\tab}{\,\,\,\,}

\newtheorem{theorem}{Theorem}[section]
\newtheorem{proposition}[theorem]{Proposition}
\newtheorem{lemma}[theorem]{Lemma}

\theoremstyle{definition}

\newtheorem{remark}[theorem]{Remark}

\title{Nernst-Planck-Navier-Stokes systems near equilibrium}

\begin{document}

\noindent {\em For Ciprian Foias, in memoriam.}

\begin{abstract}
The Nernst-Planck-Navier-Stokes system models electrodiffusion of ions in a fluid. We prove global existence of solutions in bounded domains in three dimensions with either blocking (no-flux) or uniform selective (special Dirichlet) boundary conditions for ion concentrations. The global existence of strong solutions is established for initial conditions that are sufficiently small perturbations of steady state solutions. The solutions remain close to equilbrium in strong norms. The main two steps of the proof are (1) the decay of the sum of
relative entropies (Kullback-Leibler divergences) and (2) the control of $L^2$ norms of deviations by the sum of relative entropies.
\end{abstract}

\vspace{.5cm}
\noindent

\noindent
\maketitle

\section{{Introduction}}

We study the Nernst-Planck-Navier-Stokes (NPNS) sytem, which models electrodiffusion of ions in a fluid, in the presence of boundaries. Ions suspended in a fluid are advected by the fluid flow and by an electric potential, which results from both an applied potential on the boundary and the distribution of charges carried by the ions. In addition, ionic diffusion is driven by their own concentration gradients. In turn, fluid flow is forced by the electrical field created by the ions. Such a situation is of interest from both a physical and engineering point of view, and a rigorous understanding of the underlying electrokinetic phenomena is essential to applications, which include nanofluidic devices, water filtering, and chemical mixing (see {\cite{Rub}}, {\cite{Schmuck}} for further discussions).\\
\indent The full NPNS system described above is given by
\begin{align}
(\text{Nernst-Planck equations})&\tab\partial_t c_i =\nabla\cdot(-uc_i+D_i\nabla c_i+z_i D_i c_i\nabla\Phi)\label{eq1},\tab i=1,...,N\\
(\text{Poisson-Boltzmann equation})&\tab-\varepsilon\Delta\Phi=\rho\\
(\text{Charge density})&\tab\rho=\sum_{i=1}^N z_i c_i\\
(\text{Navier-Stokes (momentum) equations})&\tab\partial_t u+u\cdot\nabla u+\nabla p=\nu\Delta u-(k_B T_K)\rho\nabla\Phi\label{NSE}\\
(\text{Divergence-free condition})&\tab\nabla\cdot u=0
\end{align}
where $c_i$ denote ionic concentrations, $z_i$ denote the corresponding valences, $u\in\mathbb{R}^3$ is fluid velocity, $p$ is pressure, $\Phi$ is electric potential, and $\rho$ is charge density. Our spatial domain is an open bounded set with smooth boundary, $\Omega\subset \mathbb{R}^3$, and we consider boundary conditions
\begin{align}
u_{|\partial\Omega}&=0\label{noslip}\\
\Phi_{|\partial\Omega}&=W\label{phi bc}
\end{align}
together with either \textbf{blocking (no-flux) boundary conditions}:
\begin{align}
(\nabla c_i+z_ic_i\nabla\Phi)_{|\partial\Omega}\cdot n&=0,\,\,\,\,\,\,i=1,...,N \label{block},
\end{align}
or with \textbf{uniform selective boundary conditions}:
\begin{align}
{c_i}_{|S_i}=\gamma_i, \,\,\,\,\,\, &(\nabla c_i+z_ic_i\nabla\Phi)_{|\partial\Omega\backslash S_i}\cdot n=0,\,\,\,\,\,i=1,...,M\label{unif1}\\
&(\nabla c_i+z_ic_i\nabla\Phi)_{|\partial\Omega}\cdot n=0,\,\,\,\,\,\,i=M+1,...,N\label{unif2}
\end{align}
where $\gamma_i(x)$ are time independent positive smooth functions, and $S_i\subset\partial\Omega$ are boundary portions. In this latter case, we require additionally that 
\begin{align}(\log \gamma_i(x) + z_iW(x))_{|S_i}= \log Z_i^{-1},\,\,\,\,\,\,i=1,...,M\label{W}\end{align} 
with each $Z_i>0$ constant on $S_i$. In the absence of this additional condition, we refer to the boundary conditions as ``general selective''.\\
\indent  The Dirichlet boundary conditions for $u$ and $\Phi$ correspond to no slip for the momentum and a fixed potential on the boundary, respectively. Blocking boundary conditions correspond to no penetration of ions across the boundary. In the case of general selective boundary conditions, the Dirichlet conditions for $c_i$ ($i=1,...,M$) represent boundaries (say, membranes) that allow controlled permeation of ions across portions ($S_i\subset\partial\Omega$) of the boundary.\\
\indent Above, $z_i\in\mathbb{R}$ are the ionic valences, and we require that there exist $i,j\in\{1,...,N\}$ such that $z_i<0<z_j$. The $D_i$'s are positive constant diffusivities, $\varepsilon>0$ represents (and is proportional to) the Debye length squared, $\nu>0$ is kinematic viscosity, $k_B$ is Boltzmann's constant, and $T_K$ is (absolute) temperature.  The potential $\Phi$ has been rescaled so that $\fr{k_BT_K}{e}\Phi$ is the electrical potential, where $e$ is elementary charge. Similarly $\rho$ has been rescaled so that $e\rho$ is the electrical charge density.

\beg{remark}
The Dirichlet boundary conditions above generalize the ``uniformly selective'' boundary conditions of \cite{PC} where $\gamma_i>0$ and $W_{|S_i}$ were required to be each constant. 
\end{remark}

Several analytical studies have been done for the Nernst-Planck-Poisson-Boltzmann system both coupled to and uncoupled to the Navier-Stokes equations. The uncoupled system is considered in {\cite{B1}}, {\cite{B2}}, {\cite{Choi}}, where existence and long-time asymptotics are studied. The coupled system in two dimensions is considered in {\cite{Sc}} where Robin boundary conditions for the electric potential are considered, and global existence and stability are shown. In {\cite{Schmuck}}, global existence of weak solutions is shown in two and three dimensions for homogeneous Neumann boundary conditions on the potential. In {\cite{Ryham}}, homogeneous Dirichlet boundary conditions on the potential are considered, and global existence of weak solutions is shown in two dimensions for large initial data and in three dimensions for small initial data (small perturbations and small initial charge). In {\cite{Wang}}, the authors study the system coupled to compressible fluid flow in the whole space $\mathbb{R}^3$ and obtain local well-posedness for large data and global well-posedness and time-decay rates for small data.\\
\indent This paper builds upon the work done in {\cite{PC}}, where in particular global existence of strong solutions to the NPNS system in bounded domains was established for two dimensions with large initial data, not just for blocking boundary conditions but also for selective (Dirichlet) boundary conditions on the concentrations, and (inhomogeneous) Dirichlet data for the potential. Given that the Navier-Stokes equations form a subsystem of the NPNS system, analogous results in three dimensions are currently out of reach. The difficulty in 3D does not rest only with the Navier-Stokes problem. The global existence of smooth solutions to the Nernst-Planck system without fluid or with fluid obeying zero Reynolds number equations (Stokes flow) is in general open. In this paper, we consider the full system without restricting to Stokes flow, and we prove global existence of strong solutions in three dimensions for small perturbations away from equilibrium, with either blocking or uniform selective boundary conditions for ion concentrations and Dirichlet boundary conditions for the electric potential.

\section{{Preliminaries}}

\subsection{{Boltzmann Steady States. Poisson-Boltzmann Equations.}}

The \textbf{Boltzmann (steady) states} are defined as
\begin{align}
c_i^*(x)=(Z_i)^{-1}e^{-z_i\Phi^*(x)} \label{boltz def}
\end{align}
where $Z_i>0$ are constants, possibly depending on $\Phi^*$. The function $\Phi^*$ is the solution to the semilinear elliptic equation
\begin{align}
-\varepsilon\Delta\Phi^*=\rho^* \label{boltz1}
\end{align}
with
\begin{align}
\rho^*=\Sigma_{i=1}^N z_i c_i^* \label{boltz2}
\end{align}
and boundary condition (\ref{phi bc}). Given a trace $W\in W^{\frac{3}{2},p}(\partial\Omega)$, the existence and uniqueness of solutions to (\ref{boltz def})-(\ref{boltz2}) (the Poisson-Boltzmann equations) in $W^{2,p}(\Omega)$ is known (see e.g. {\cite{PC}}) and follows from classical semilinear elliptic theory, using variational methods. In particular, $\Phi^*\in W^{1,\infty}(\Omega)$ if $p>3$.\\
\indent We note that $c_i^*,\,\Phi^*$, together with $u\equiv 0$, are steady state solutions of the Nernst-Planck-Navier-Stokes system.
\begin{remark}
In the introduction we required that there exist $i,j\in\{1,...,N\}$ such that $z_i<0<z_j$. This condition is required to establish boundedness properties of $\Phi^*$ (see {\cite{PC}}). This condition does not play a further role in our analysis.
\end{remark}
\begin{remark}
From the definition, it is clear that $c_i^*$, $i=1,...,N$ are strictly positive quantities. Since these variables represent ion concentrations, their nonnegativity is necessary for them to carry physical meaning. From the equations of the NPNS system, it may not be immediately obvious that $c_i$, $i=1,...N$, will remain nonnegative given $c_i(x,0)\geq 0$. However, they do indeed remain nonnegative for as long as they are regular, as shown in {\cite{PC}}. Thus, nonnegativity of these quantities are assumed throughout in this paper.
\end{remark}

\subsection{{Energy Functional}}
The following energy functional was introduced in \cite{PC}.
\begin{align}
\mathcal{E}=\mathcal{E}(c_i,\,\Phi;\,c_i^*,\,\Phi^*)=\int_\Omega\left[\sum_{i=1}^N E_i c_i^*+\frac{\varepsilon}{2}|\nabla(\Phi-\Phi^*)|^2\right]\,dx\label{energy}
\end{align}
where $E_i$ is defined as
\begin{align}
E_i=\frac{c_i}{c_i^*}\log\left(\frac{c_i}{c_i^*}\right)-\frac{c_i}{c_i^*}+1.
\end{align}
The term $\sum_i E_i c_i^*$ is the sum of relative entropies (or Kullback-Leibler divergences) relative to fixed Boltzmann states. Relative entropy is a common tool in the literature of probability and information theory; its use in the analysis of scalar PDEs (e.g. Fokker-Planck equations) is also well-established. The use of a sum of relative entropies is not common.

Below we state a slight generalization of the result of {\cite{PC}} that gives the NPNS system a dissipative structure and is the main ingredient for controlling growth of solutions.
\begin{theorem}\label{dissipation theorem}
Given a solution to the NPNS system with blocking or uniform selective boundary conditions, let $\mathcal{E}$ be defined with arbitrary $Z_i>0$, $i=1,...,N$ in the blocking case and with constant $Z_i^{-1}=(\gamma_i(x) e^{z_i W(x)})_{|\pa\Omega}$, $i=1,...,M$ and arbitrary $Z_i>0$, $i=M+1,...,N$ in the uniform selective case. Then the relation
\begin{align}
\frac{d}{dt}\left[\frac{1}{2k_B T_K}\int_\Omega |u|^2\,dx+\mathcal{E}\right]=-\mathcal{D}-\frac{\nu}{k_B T_K}\int_\Omega|\nabla u|^2\,dx\leq 0           \label{dissipation}
\end{align}
holds for all $t>0$, where 
\begin{align}
\mathcal{D}=\sum_{i=1}^N D_i\int_\Omega c_i\left|\nabla\frac{\delta \mathcal{E}}{\delta c_i}\right|^2\,dx.
\end{align}
\end{theorem}
The quantities $\frac{\delta \mathcal{E}}{\delta c_i}$ are densities of the first variations of $\mathcal{E}$. A  computation shows that
\begin{align}
\frac{\delta \mathcal{E}}{\delta c_i}=\log\left(\frac{c_i}{c_i^*}\right)+z_i(\Phi-\Phi^*)=\log c_i+z_i\Phi+\log Z_i\label{frechet}
\end{align}
where the second equality follows from the definition of $c_i^*$ (\ref{boltz def}).
\begin{remark}
The choice of normalizing constants $Z_i^{-1}=(\gamma_i(x) e^{z_i W(x)})_{|\pa\Omega}$ in the case of uniform selective boundary conditions ensures that $c_i^*$, defined with this normalizing constant, satisfies the prescribed Dirichlet boundary conditions (\ref{unif1}).
\end{remark}
\begin{proof} Let $\Gamma_i$ be smooth time-independent positive functions, and let $\Phi^*$ be a smooth time independent function that obeys
\be
\Phi^*_{|\pa\Omega} = W.
\la{phistarbc}
\ee
Let us denote by $\mu_i$ the electro-chemical potential
\be
\mu_i = \log  c_i + z_i\Phi
\la{mui}
\ee
and by $\mu^*_i$ its analogue
\be
\mu^*_i  = \log \Gamma_i +z_i\Phi^*.
\la{mustari}
\ee
We note that the Nernst-Planck equations read
\be
(\pa_t + u\cdot\na)c_i = D_i\div(c_i\na \mu_i).
\la{npmui}
\ee
We multiply each of the equations (\ref{npmui}) by 
\be
\log\left (\fr{c_i}{\Gamma_i}\right) + z_i(\Phi-\Phi^*) = \mu_i-\mu^*_i
\la{mudif}
\ee
and integrate in $\Omega$. On the right hand side we have
\be
D_i\int_{\Omega}\div(c_i\na \mu_i)(\mu_i-\mu_i^*) = 
-D_i\int_{\Omega}c_i\na\mu_i\cdot\na(\mu_i-\mu_i^*)dx + D_i\int_{\pa\Omega}c_i(\mu_i-\mu_i^*)(n\cdot\na \mu_i)dS.
\la{rhsi}
\ee
Blocking boundary conditions are precisely $n\cdot\na \mu_i =0$. 
If we select $\Gamma_i$ such that 
\be
{\Gamma_i(x)}_{|S_i} = {\gamma_i(x)}_{| S_i}
\la{gammaibc}
\ee
for $i=1, \dots , M$ in the case of selective boundary conditions, then, in view of the fact that $\Phi-\Phi^*$ vanish at the boundary and the fact that $\log\left(\fr{c_i}{\Gamma_i}\right)$ vanish at $S_i$ for $i=1,\dots, M$, we have that the boundary contribution in (\ref{rhsi}) vanishes. 
We have not used the condition (\ref{W}), nor did we use any relation between $\Phi^*$ and $\Gamma_i$. We have thus, for general selective or blocking boudary conditions
\be
\ba
D_i\int_{\Omega}\div(c_i\na \mu_i)(\mu_i-\mu_i^*)  = -D_i\int_{\Omega}c_i\na\mu_i\cdot\na(\mu_i-\mu_i^*)dx \\
\le - \fr{D_i}{2}\int_{\Omega}c_i|\na\mu_i|^2dx + \fr{D_i}{2}\int_{\Omega}c_i |\na\mu^*_i|^2dx.
\ea
\la{rhsib}
\ee 
In order to compute the left hand side we observe that
\be
((\pa_t + u\cdot\na )c_i) \left(\log\left(\fr{c_i}{\Gamma_i}\right)\right) =
(\pa_t + u\cdot\na )\left(c_i\log\left(\fr{c_i}{\Gamma_i}\right) -c_i\right) +
c_iu\cdot\na\log\Gamma_i
\la{dtlog}
\ee
and thus we have, after summing in $i$
\be
\ba
\sum_{i=1}^N\int_{\Omega} ((\pa_t + u\cdot\na)c_i)(\mu_i-\mu_i^*)dx \\
= \sum_{i=1}^N\int_{\Omega}\left[(\pa_t + u\cdot\na )\left(c_i\log\left(\fr{c_i}{\Gamma_i}\right) -c_i\right) + c_iu\cdot\na\log\Gamma_i + ((\pa_t+u\cdot\na)\rho)(\Phi-\Phi^*)\right]
\ea
\la{lhsa}
\ee
where we used $\sum_{i=1}^N z_ic_i = \rho$.  We introduce $\rho^*$ defined for the purpose of this proof as
\be
\rho^* = -\varepsilon\D\Phi^*
\la{rhostarphistar}
\ee
without any connection to $\Gamma_i$. Then we note that
\be
(\pa_t\rho)(\Phi-\Phi^*) = (\pa_t(\rho-\rho^*))(\Phi-\Phi^*).
\la{rhotph}
\ee
Now $\rho-\rho^* = -\varepsilon\D(\Phi-\Phi^*)$ by (\ref{rhostarphistar}), and $\Phi-\Phi^*$ vanishes at the boundary by (\ref{phistarbc}), and therefore, from (\ref{lhsa}) we obtain
\be
\ba
\sum_{i=1}^N\int_{\Omega} ((\pa_t + u\cdot\na)c_i)(\mu_i-\mu_i^*)dx \\
= \fr{d}{dt}\int_{\Omega}\left[\sum_{i=1}^N\left(c_i\log\left(\fr{c_i}{\Gamma_i}\right) -c_i\right) + \fr{\varepsilon}{2}|\na(\Phi-\Phi^*)|^2\right]dx\\
+\int_{\Omega}\sum_{i=1}^Nc_iu\cdot\na\log(\Gamma_i + z_i\Phi^*) dx +
\int_{\Omega} \left[(u\cdot\na)\rho)\Phi -(\rho u\cdot\na\Phi^*   +\Phi^*u\cdot\na\rho)\right]
\ea
\la{lhsb}
\ee
where we added and subtracted $\sum_{i=1}^N c_iu\cdot\na z_i\Phi^*= \rho u\cdot\na\Phi^*$. In view of the fact that $\rho u\cdot\na\Phi^*   +\Phi^*u\cdot\na\rho = \na\cdot(u\rho\Phi^*)$ and the fact that $u\cdot n$ vanishes at the boundary,the left hand side is 
\be
\ba
\sum_{i=1}^N\int_{\Omega} ((\pa_t + u\cdot\na)c_i)(\mu_i-\mu_i^*)dx \\
= \fr{d}{dt}\int_{\Omega}\left[\sum_{i=1}^N\left(c_i\log\left(\fr{c_i}{\Gamma_i}\right) -c_i\right) + \fr{\varepsilon}{2}|\na(\Phi-\Phi^*)|^2\right]dx\\
+\int_{\Omega}\sum_{i=1}^Nc_iu\cdot\na\mu_i^*dx +
\int_{\Omega} (u\cdot\na)\rho)\Phi dx.
\ea
\la{lhsc}
\ee
Putting together the sum of (\ref{rhsib}) and (\ref{lhsc}) we have
\be
\ba
\fr{d}{dt}\mathcal E_0 = \int_{\Omega}\rho u\cdot\na\Phi dx  -\sum_{i=1}^ND_i\int_{\Omega}c_i\na\mu_i\cdot\na(\mu_i-\mu_i^*)dx 
- \int_{\Omega}\sum_{i=1}^Nc_iu\cdot\na \mu_i^*dx\\
\le  \int_{\Omega}\rho u\cdot\na\Phi dx  
  - \sum_{i=1}^N \fr{D_i}{2}\int_{\Omega}c_i|\na\mu_i|^2dx + \sum_{i=1}^N\fr{D_i}{2}\int_{\Omega}c_i |\na\mu^*_i|^2dx - \int_{\Omega}\sum_{i=1}^Nc_iu\cdot\na\mu_i^*dx.
\ea
\la{ezb}
\ee
where 
\be
\mathcal E_0 = \int_{\Omega}\left[\sum_{i=1}^N\left(c_i\log\left(\fr{c_i}{\Gamma_i}\right) -c_i\right) + \fr{\varepsilon}{2}|\na(\Phi-\Phi^*)|^2\right]dx
\la{ezero}
\ee
We have used only the facts that ${\Gamma_i}_{|S_i} = \gamma_i$ for $i=1, \dots M$, i.e. (\ref{gammaibc}) and (\ref{phistarbc}). Now the term 
$\int_{\Omega}\rho u\cdot\na\Phi dx $ is precisely the term needed to cancel the work of electrical forces in the Navier-Stokes energy balance. We obtain
\be
\ba
\fr{d}{dt}\left[\fr{1}{2k_BT_K}\int_{\Omega}|u|^2dx + \mathcal E_0\right] =
-\fr{\nu}{k_BT_K}\int_{\Omega}|\na u|^2dx -\sum_{i=1}^N D_i\int_{\Omega}c_i\na\mu_i\cdot\na(\mu_i-\mu^*_i) dx\\ -\int_{\Omega}\sum_{i=1}^Nc_iu\cdot\na\mu_i^*dx\\
\le -\fr{\nu}{k_BT_K}\int_{\Omega}|\na u|^2dx -\sum_{i=1}^N \fr{D_i}{2}\int_{\Omega}c_i|\na\mu_i|^2dx\\
 + \sum_{i=1}^N\fr{D_i}{2}\int_{\Omega}c_i |\na\mu^*_i|^2dx - \int_{\Omega}\sum_{i=1}^Nc_iu\cdot\na\mu_i^*dx.
\ea
\la{enezerob}
\ee
This inequality is true for any choices of $\Gamma_i$ with ${\Gamma_i(x)}_{|S_i} = \gamma_i(x)$  for $i=1, \dots, M$, and $\Phi^*$ with $\Phi^*(x)_{|\pa\Omega}= W(x)$. No relation (\ref{W}) is needed, nor is the Poisson-Boltzmann equation (\ref{boltz1}) required. We note that
\be
\fr{\delta\mathcal E_0}{\delta c_i} = \mu_i-\mu_i^*.
\la{delezero}
\ee
In the case of uniform selective boundary conditions we put
\be
\Gamma_i = Z_i^{-1}e^{-z_i\Phi^*}
\la{gammaiphistar}
\ee
and observe that the condition (\ref{W}) implies that ${\Gamma_i}_{| S_i} = \gamma_i.$, i.e. (\ref{gammaibc}) holds.  In this case 
\be
\mu^*_i = \log Z_i^{-1}
\la{mustarzi}
\ee
are constant in space, $\na \mu^*_i =0$,  and thus
\be
\fr{d}{dt}\left[\fr{1}{2k_BT_K}\int_{\Omega}|u|^2dx + \mathcal E_0\right] = -\fr{\nu}{k_BT_K}\int_{\Omega}|\na u|^2dx -\sum_{i=1}^N D_i\int_{\Omega}c_i|\na \mu_i|^2dx.
\la{enedecay}
\ee
Remarkably, the Poisson-Boltzmann equation is not needed. If it is satisfied, then $\Gamma_i = c^*_i$. This concludes the proof. \end{proof}

\subsection{Local Existence}
We state below the local existence result given in {\cite{PC}}.
\begin{theorem}\label{local}
Let $\Omega\subset\mathbb{R}^d$, $d=2,3$ be an open bounded domain with smooth boundary. Let $z_i\in\mathbb{R}$, $1\leq i\leq N$ and let $\varepsilon>0$, $D_i>0$, $i=1,...,N$. Let $p=2q>2d$. Then for $c_i(0)\geq 0$  given in $W^{2,q}(\Omega)$, $i=1,...,N$, $W\in W^{\frac{3}{2},p}(\partial\Omega)$, and $u_0\in (W^{1,p}(\Omega))^d\cap\{\nabla\cdot u=0\}$, there exists $T_0>0$ depending only on the parameters of the of problem $\varepsilon, D_i, z_i, \nu, \Omega$, the initial energy $\mathcal{E}(t=0)$ and on the norms $\norm{c_i(0)}_p$, $\norm{W}_{W^{\frac{3}{2},p}}$, $\norm{u_0}_{W^{1,2}}$ such that there exists a unique strong solution of the NPNS system in $\Omega\times[0,T_0)$ for either blocking or uniform selective boundary conditions, satisfying
\begin{align*}
\sup_{0\leq t< T_0}\norm{c_i(t)}_p\leq 3\norm{c_i(0)}_p.
\end{align*}
\end{theorem}
In two dimensions, it suffices to establish uniform bounds for $\norm{c_i}_p$ to obtain global existence. In three dimensions, in addition to uniform bounds for $\norm{c_i}_p$, we must verify or impose certain smallness conditions on initial velocity and forcing to guarantee global regularity for the Navier-Stokes subsystem.

\section{{Global Existence for Small Perturbations (Blocking)}}

In this section we consider blocking boundary conditions and establish global existence of strong solutions for the NPNS system under a small perturbation condition. For initial conditions close enough to the steady state solutions (Boltzmann states + fluid at rest), we control the solutions for long time to guarantee that blow up does not occur and the local existence and uniqueness theorem guarantees global existence.

\subsection{The set-up}

As stated in Theorem \ref{dissipation theorem}, for blocking boundary conditions, the dissipation relation (\ref{dissipation}) holds for arbitrary normalizing constants $Z_i>0$ in the definition of $c_i^*$. We exploit this freedom by selecting
\begin{align}
Z_i=\left(\int_\Omega c_i(x,0)\,dx\right)^{-1}\int_\Omega e^{-z_i\Phi^* }\,dx.
\end{align}
We observe that from (\ref{eq1}), (\ref{noslip}) and (\ref{block}), we have
\begin{align}
\frac{d}{dt}\int_\Omega c_i(x,t)\,dx=0\Rightarrow \int_\Omega c_i(x,t)\,dx=\int_\Omega c_i(x,0)\,dx,\, t\geq 0.\label{conserved}
\end{align}
Then, as a consequence for our choice of $Z_i$, we have
\begin{align*}
\int_\Omega (c_i(x,t)-c_i^*(x))\,dx=0,\,t\geq 0.
\end{align*}
This relation justifies the application of Poincar\'{e}'s inequalities and certain Gagliardo-Nirenberg interpolation inequalities to the function $c_i-c_i^*$. Similar inequalities will be used for the function $\nabla(\Phi-\Phi^*)$, whose integral also vanishes due to Dirichlet boundary conditions.

The choice of $Z_i$ yields the correct Boltzmann states because in the blocking case, ionic concentrations are conserved (\ref{conserved}), and we expect long-time behavior $c_i\to c_i^*$ (see \cite{PC}).

For future reference, we state below the consequence of the dissipation relation (\ref{dissipation}):
\begin{align}
\frac{1}{2k_BT_K}\norm{u(t)}_2^2+\mathcal{E}(t)&\leq \frac{1}{2k_BT_K}\norm{u_0}_2^2+\mathcal{E}(0),\,t\geq 0.\label{diss2}
\end{align}
As we will frequently refer to the quantity on the right hand side of the inequality, we shall label it 
\begin{align}
E_K:=\frac{1}{2k_BT_K}\norm{u_0}_2^2+\mathcal{E}(0).\end{align}

Now we may state the main result which gives global existence in three dimensions for small perturbations.

\begin{theorem}\label{theorem}
Let $\Omega\subset\mathbb{R}^3$ be an open bounded domain with smooth boundary. Let $z_i\in\mathbb{R}$, $i=1,...,N$ (such that there exist $i,j$ with $z_i<0<z_j$), and let $\varepsilon>0$, $D_i>0$, $i=1,...,N$. Let $p=2q>6$, and suppose the following initial data are given: $c_i(0)\geq 0$  in $W^{2,q}(\Omega)$, $i=1,...,N$, $W\in W^{\frac{3}{2},p}(\partial\Omega)$, and $u_0\in (W_0^{1,p}(\Omega))^3\cap\{\nabla\cdot u=0\}$. Furthermore, assume that the following smallness conditions for the initial data are satisfied:
\begin{align*}
E_K&\leq \min\{\delta_1,\delta_3\}\\
\sum_{i=1}^N\norm{c_i(0)-c_i^*}_2^2&\leq \delta_2\\
\norm{\nabla u_0}_2&\leq \zeta_1
\end{align*} 
where the constants $\delta_1, \delta_2, \delta_3, \zeta_1$ (see (\ref{deltaone}), (\ref{deltatwo}), (\ref{deltathree})) depend on the parameters of the problem, boundary data and initial concentrations.
Then there exists a unique strong solution of the NPNS system (\ref{eq1})-(\ref{phi bc}) with blocking boundary conditions (\ref{block}) in $\Omega\times[0,\infty)$. 
\end{theorem}
\subsection{Proof of Theorem \ref{theorem}}
The proof follows from Theorem \ref{local} and the a priori uniform estimates proven below. Our primary and ultimate goal is to uniformly control $c_i$ in $L^p$ ($p>6$).
\subsubsection{Uniform $L^\infty(L^2)$ bounds for $c_i$. Smallness condition. Bounds on $\Phi$.}
We employ a Gr\"{o}nwall-type argument for the quantity $\norm{c_i-c_i^*}_2^2$.

Starting from (\ref{eq1}) and using the time independence of $c_i^*$ and the relation
\begin{align}
\nabla c_i^*=-z_ic_i^*\nabla\Phi^*
\end{align}
which follows from the definition (\ref{boltz def}), we obtain, after adding and subtracting like terms, the following:
\begin{align}
	\partial_t(c_i - c_i^*) = &\nabla\cdot(-u(c_i -c_i^*) -uc_i^* +D_i \nabla(c_i-c_i^*)+z_iD_i(c_i-c_i^*)\nabla(\Phi-\Phi^*)\nonumber\\
													&+z_iD_ic_i^*\nabla(\Phi-\Phi^*)+z_iD_i(c_i-c_i^*)\nabla\Phi^*).\label{L2}
\end{align}
Multiplying (\ref{L2}) by $c_i-c_i^*$ and integrating by parts, we obtain
\begin{align}
\frac{1}{2}\frac{d}{dt}\int_\Omega|c_i-c_i^*|^2\,dx+D_i\int_\Omega|\nabla(c_i-c_i^*)|^2\,dx &= I_1 + I_2 + I_3 + I_4 + I_5\label{L2-1}
\end{align}
where
\begin{align*}
	I_1 &=\int_\Omega (c_i-c_i^*)u\cdot\nabla(c_i-c_i^*)\,dx\\
	I_2 &=\int_\Omega c_i^*u\cdot\nabla(c_i-c_i^*)\,dx\\
	I_3&=-\int_\Omega z_iD_i(c_i-c_i^*)\nabla(\Phi-\Phi^*)\cdot\nabla(c_i-c_i^*)\,dx\\
	I_4&=-\int_\Omega z_iD_i c_i^*\nabla(\Phi-\Phi^*)\cdot\nabla(c_i-c_i^*)\,dx\\
	I_5&=-\int_\Omega z_iD_i (c_i-c_i^*)\nabla\Phi^*\cdot\nabla(c_i-c_i^*)\,dx
\end{align*}
No boundary terms occur due to no-slip and blocking boundary conditions. Now we bound each term using elliptic regularity and H\"{o}lder, Young's, and interpolation inequalities:
\begin{align*}
I_1=&\frac{1}{2}\int_\Omega u\cdot\nabla(c_i-c_i^*)^2\,dx=-\frac{1}{2}\int_\Omega(\nabla\cdot u)(c_i-c_i^*)^2\,dx=0\\
I_2\leq&\norm{c_i^*}_\infty\norm{u}_2\norm{\nabla(c_i-c_i^*)}_2\\
		\leq&\epsilon' D_i\norm{\nabla(c_i-c_i^*)}_2^2+C_{2,i}\frac{\norm{c_i^*}_\infty^2}{D_i}\norm{u}_2^2\\
I_3\leq&D_i|z_i|\norm{c_i-c_i^*}_3\norm{\nabla(\Phi-\Phi^*)}_6\norm{\nabla(c_i-c_i)}_2\\
		\leq&C_{3,i}D_i|z_i|\varepsilon^{-1}\norm{\nabla(c_i-c_i^*)}_2^\frac{3}{2}\norm{c_i-c_i^*}_2^\frac{1}{2}\norm{\rho-\rho^*}_2\\
		\leq&C_{3,i}D_i\max_j|z_j|^2\varepsilon^{-1}\norm{\nabla(c_i-c_i^*)}_2^\frac{3}{2}(\Sigma_{j=1}^N \norm{c_j-c_j^*}_2^2)^\frac{3}{4}\\
		\leq&\epsilon' D_i\norm{\nabla(c_i-c_i^*)}_2^2+C_{3,i}\frac{D_i\max_j|z_j|^8}{\varepsilon^4}(\Sigma_{j=1}^N\norm{c_j-c_j^*}_2^2)^3\\
I_4\leq&D_i|z_i|\norm{c_i^*}_\infty\norm{\nabla(\Phi-\Phi^*)}_2\norm{\nabla(c_i-c_i^*)}_2\\
		\leq&\epsilon'D_i\norm{\nabla(c_i-c_i^*)}_2^2+C_{4,i}D_i|z_i|^2\norm{c_i^*}_\infty^2\norm{\nabla(\Phi-\Phi^*)}_2^2\\
I_5\leq&D_i|z_i|\norm{\nabla\Phi^*}_\infty\norm{c_i-c_i^*}_2\norm{\nabla(c_i-c_i^*)}_2\\
		\leq&C_{5,i}D_i|z_i|\norm{\nabla\Phi^*}_\infty\norm{c_i-c_i^*}_1^\frac{2}{5}\norm{\nabla(c_i-c_i^*)}_2^\frac{8}{5}\\
		\leq&\epsilon'D_i\norm{\nabla(c_i-c_i^*)}_2^2+C_{5,i}D_i|z_i|^5\norm{\nabla\Phi^*}_\infty^5\norm{c_i-c_i^*}_1^2
\end{align*}
Above, the constants $C_{j,i}$, $j=2,3,4,5$, which may differ from line to line, are ultimately nondimensional.
We take $\epsilon'=1/8$ and reutrn to (\ref{L2-1}), resulting in the following differential inequality:
\begin{align}
\frac{d}{dt}\sum_{i=1}^N\norm{c_i-c_i^*}_2^2+D^-\sum_{i=1}^N\norm{\nabla(c_i-c_i^*)}_2^2\leq F(t)+H\left(\sum_{i=1}^N\norm{c_i-c_i^*}_2^2\right)^3\label{L2-2}
\end{align}
where
\begin{align}
D^-&=\min_i\{D_i\}\label{c1}\\
D^+&=\max_i\{D_i\}\\
\tilde{C}&=\max_{j,i}\{C_{j,i}\}\label{clast}\\
z&=\max_i|z_i|\\
F(t)&=2N\tilde{C}\left(\beta_1\norm{u}_2^2+\beta_2\norm{\nabla(\Phi-\Phi^*)}_2^2+\beta_3\max_i \norm{c_i-c_i^*}_1^2\right)\label{F}\\
H&=2N\tilde{C}\frac{D^+ z^8}{\varepsilon^4}\label{H}\\
\beta_1&=\frac{\max_i\norm{c_i^*}_\infty^2}{D^-}\\
\beta_2&=D^+z^2\max_i\norm{c_i^*}_\infty^2\\
\beta_3&=D^+z^5\norm{\nabla\Phi^*}_\infty^5.\\
\end{align}
One final application of Poincar\'{e}'s inequality on (\ref{L2-2}) gives us the form of the inequality that we will work with:
\begin{align}
\frac{dw}{dt}\leq F(t)-C_{\Omega}D^- w +Hw^3 \label{L2-3}
\end{align}
where we have set $w(t):=\sum_{i=1}^N\norm{c_i(t)-c_i^*}_2^2$, and $C_\Omega$ is a constant depending on the geometry of $\Omega$, with dimensions of inverse length squared.\\ 
\indent Looking at the terms that comprise $F(t)$, we see that if we are able to bound ${\max_i\norm{c_i(t)-c_i^*}_1^2}$ above by a constant multiple of $E_K$, then $F(t)$  itself may be bounded above by a constant multiple of $E_K$. Indeed, we have the following inequality.
\begin{lemma}\label{lemma}(Csiszar-Kullback inequality, see e.g. {\cite{Cover}})
Let $f,g\in L^1(\Omega)$ satisfy $f\geq 0$, $g>0$ and $\int_\Omega f\,dx=\int_\Omega g\,dx=\alpha$. Then
\begin{align}
\norm{f-g}_1^2\leq C\alpha\int_\Omega \left(f\log\left(\frac{f}{g}\right)-f+g\right)\,dx
\end{align}
for $C>0$ independent of $f,g$. 
\end{lemma}
Taking $f=c_i$ and $g=c_i^*$ from the lemma, we have
\begin{align}
\max_i\norm{c_i-c_i^*}_1^2\leq C\max_i\norm{c_i(0)}_1\max_i\int_\Omega E_ic_i^*\,dx\leq C\max_i\norm{c_i(0)}_1 E_K\label{L1}.
\end{align}
Thus, we have the following bound that follows from the definition of $F(t)$ and from (\ref{diss2}), (\ref{L1}):
\begin{align}
F(t)\leq& \widetilde{F}\cdot E_K:=2N\tilde{C}\max\{2\beta_1 k_B T_K, 2\beta_2/\varepsilon, C\beta_3\max_i\norm{c_i(0)}_1\}E_K\label{Ftil}
\end{align}
uniformly in time. 
\begin{proposition}
Let $w(t)=\sum_{i=1}^N\norm{c_i(t)-c_i^*}_2^2$. Suppose the following smallness conditions are satisfied:
\begin{align}
E_K\leq& \delta_1\label{small1},\\
w(0)\leq& \delta_2\label{small2}
\end{align}
where the constants $\delta_1$ and $\delta_2$ are given below in (\ref{deltaone}), and (\ref{deltatwo}). Then $w(t)$ is bounded above uniformly in time
 by $\tilde{w}$, given below in (\ref{inter}).
\end{proposition}
\begin{proof}
As long as
\be
w(t)\le\sqrt{\fr{C_\Omega D^-}{2H}} = \tilde{w}
\la{inter}
\ee
we have from (\ref{L2-3}) and (\ref{Ftil})
\be
\fr{dw}{dt}\le \widetilde{F} E_K- \lambda w
\la{interm}
\ee
with
\be
\lambda = \fr{C_{\Omega}D^{-}}{2}
\la{lambda}
\ee
which results in 
\be
w(t) \le w(0)e^{-\lambda t} +  \fr{\widetilde{F}}{\lambda} E_K,
\la{interme}
\ee
and therefore, if
\be
w(0) + \fr{\widetilde{F}}{\lambda} E_K< \sqrt{\fr{C_{\Omega} D^-}{2H}}
\la{condperp}
\ee
then (\ref{inter}) holds for all time. This is achieved for instance if
\be
E_K\le \fr{1}{8\sqrt{2}} \left(C_{\Omega} D^-\right)^{\fr{3}{2}} H^{-\fr{1}{2}}\widetilde{F}^{-1} = \delta_1
\la{deltaone}
\ee
and 
\be
w(0)\le \fr{1}{4\sqrt{2}}\left(C_{\Omega}D^-\right)^{\fr{1}{2}}H^{-\fr{1}{2}} = \delta_2.
\la{deltatwo}
\ee

\end{proof}

Two consequences of the uniform $L^\infty(L^2)$ bounds on $c_i$, which follow from Sobolev imbeddings and elliptic regularity, are
\begin{align}
\norm{\nabla\Phi(t)}_6&\leq\Gamma_1\label{G1}\\
\norm{\Phi(t)}_\infty&\leq\Gamma_2\label{G2}
\end{align}
where $\Gamma_1$ and $\Gamma_2$ are independent of time.

\subsubsection{Uniform $L^\infty(L^p)$ bounds for $c_i$ for $1\leq p\leq \infty$}\label{lp bounds}
Using the fact that $c_i\in L^\infty(L^2)$, it is possible to inductively establish uniform $L^\infty(L^p)$ for all $p$, both finite and infinite. The following Moser-type iteration was done in {\cite{Sc}} and {\cite{Choi}}.\\
\indent For $k=2,3,,...$, multiply (\ref{eq1}) by $c_i^{2k-1}$ and integrate by parts to obtain
\begin{align}
\frac{1}{2k}\frac{d}{dt}\int_\Omega c_i^{2k}\,dx=&-(2k-1)\int_\Omega-c_i^{2k-1}(u\cdot\nabla c_i)+D_i c_i^{2k-2}|\nabla c_i|^2+D_iz_i c_i^{2k-1}(\nabla\Phi\cdot\nabla c_i)\,dx\nonumber\\
=&-\frac{2k-1}{k^2}D_i\int_\Omega |\nabla c_i^{k}|^2\,dx-\frac{2k-1}{k}D_i\int_\Omega z_i c_i^{k}(\nabla\Phi\cdot\nabla c_i^{k})\,dx. \label{Lp}
\end{align}
Now we estimate the second integral on the right by interpolation,
\begin{align*}
\int_\Omega z_ic_i^k(\nabla\Phi\cdot\nabla c_i^k)\,dx\leq&|z_i|\norm{\nabla\Phi}_6\norm{\nabla c_i^k}_2\norm{c_i^k}_3\\																			
																									\leq&\epsilon\norm{\nabla c_i^k}_2^2+C_k\norm{c_i^k}_2^2
\end{align*}
where $C_k$ is a constant depending on $\epsilon, z, \Gamma_1$, and $\Omega$, with dimensions of inverse length squared. Thus returning to (\ref{Lp}), we see that for an appropriate choice of $\epsilon$, we end up with
\begin{align}
\frac{d}{dt}\norm{c_i^k}_2^2+D_i\norm{\nabla c_i^k}_2^2\leq C(k)D_id_\Omega^{-2}\norm{c_i^k}_2^2\label{Lp-1}
\end{align}
where we have factored out $d_\Omega^{-2}$ ($d_\Omega:=$ diameter of $\Omega$) from the coefficient to make $C(k)$ nondimensional. It is important to note that the manipulations (a finite number of applications of multiplying/adding rational functions of $k$) leading to the constant $C(k)>0$ are such that $C(k)$ displays at most polynomial growth with respect to $k$. In particular, there exist nondimensional constants $C',m>0$ depending on $\Omega$, $z$, and $\Gamma_1$ but independent of $k$ such that $C(k)\leq C'k^m$ for all $k\geq 2$. To go one step further, we may add, say, $k^m$ to $C(k)$ without affecting the direction of the inequality in (\ref{Lp-1}) so that 
\begin{align}
k^m\leq C(k)\leq C''k^m,\,\,\,\,\,\, (C''=C'+1)\label{poly}.
\end{align}
Next, again from interpolation, we have
\begin{align}
\norm{c_i^k}_2^2\leq \epsilon d_\Omega^2\norm{\nabla c_i^k}_2^2+C_\epsilon d_\Omega^{-3}\norm{c_i^k}_1^2.\label{Lp-2}
\end{align}
Then for appropriate $\epsilon=\epsilon(k)$, (\ref{Lp-1}) and (\ref{Lp-2}) give us
\begin{align}
\frac{d}{dt}\norm{c_i^k}_2^2\leq -D_i d_\Omega^{-2}\norm{c_i^k}_2^2+C(k)D_i d_\Omega^{-5}\norm{c_i^k}_1^2\label{Lp-3}
\end{align}
where $C(k)$ is modified from before, but still exhibits at most polynomial growth with respect to $k$ (and without loss of generality still satisfies (\ref{poly})).\\
\indent Now we define 
\begin{align}
S_k:=\max\{d_\Omega^{3}\norm{c_i(0)}_\infty,d_\Omega^{3-\frac{3}{k}}\sup_{t\geq 0}\norm{c_i}_k\},\tab k=1,2,3,...\label{S}
\end{align}
which may, a priori, allow for infinite values.\\
\indent To proceed by induction, we assume that for some $k$, $S_k$ is finite. Then from (\ref{Lp-3}), a Gr\"{o}nwall argument with integrating factor $e^{tD_id_\Omega^{-2}}$ gives
\begin{align}
&\norm{c_i^{k}}_2^2\leq e^{-tD_id_\Omega^{-2}}\norm{c_i^{k}(0)}_2^2+C(k)d_\Omega^{-6k+3}S_k^{2k}\leq d_\Omega^3\norm{c_i(0)}_\infty^{2k}+C(k)d_\Omega^{-6k+3}S_k^{2k}\nonumber\\
\Rightarrow\,&d_\Omega^{6k-3}\norm{c_i}_{2k}^{2k}\leq d_\Omega^{6k}\norm{c_i(0)}_\infty^{2k}+C(k)S_k^{2k}\leq d_\Omega^{6k}\norm{c_i(0)}_\infty^{2k}+C''k^m S_k^{2k}\leq 2C''k^mS_k^{2k}\nonumber\\
\Rightarrow\,&d_\Omega^{3-\frac{3}{2k}}\norm{c_i}_{2k}\leq (2C'')^\frac{1}{2k}k^\frac{m}{2k}S_k\label{52}
\end{align}
Thus recursively, we see that $S_k<\infty$ for all finite $k,$ and since $(2C'')^\frac{1}{2k}k^\frac{m}{2k}S_k\geq d_\Omega^3\norm{c_i(0)}_\infty$ we have the relation
\begin{align*}
S_{2k}=&\max\{d_\Omega^{3}\norm{c_i(0)}_\infty,\,d_\Omega^{3-\frac{3}{2k}}\sup_{t\geq 0}\norm{c_i}_{2k}\}\\
\leq&\max\{d_\Omega^3\norm{c_i(0)}_\infty,\,(2C'')^\frac{1}{2k}k^\frac{m}{2k}S_k\}\\
=&(2C'')^\frac{1}{2k}k^\frac{m}{2k}S_k
\end{align*}
for all $k$. Setting $k=2^j$, we see that
\begin{align}
S_{2^{j+1}}\leq (2C'')^\frac{1}{2^{j+1}}2^\frac{jm}{2^{j+1}}S_{2^j}
\end{align}
and thus for any $J\in\mathbb{N}$, we have
\begin{align}
S_{2^J}\leq (2C'')^\alpha 2^\beta  S_2
\end{align}
where $\alpha:=\sum_{j=1}^\infty\frac{1}{2^{j+1}}<\infty$ and $\beta:=\sum_{j=1}^\infty\frac{jm}{2^{j+1}}<\infty$. Since we know that $S_2$ is finite, letting $J\to\infty$ completes the proof of the claim that $\norm{c_i}_\infty$ is uniformly bounded in time, and in fact we have
\begin{align}
\norm{c_i}_\infty\leq (2C'')^\alpha 2^\beta d_\Omega^{-3}S_2=:\Gamma_\infty\label{infinityb}
\end{align}
with $\Gamma_\infty$ independent of time.

\subsubsection{Global regularity for Navier-Stokes subsystem}\label{NSE bound}
For small initial data, classical results (e.g. \cite{CF}, {\cite{FK}}) show the existence of strong solutions for small inital data and small forcing for Navier-Stokes equations. 

To proceed, we first apply the Leray projection onto the momentum equation (\ref{NSE})
\begin{align}
\partial_t u+\nu Au+B(u,u)=-(k_BT_K)\mathbb{P}(\rho\nabla\Phi)=:\mathbb{P}f.\label{NSE2}
\end{align}
Above $\mathbb{P}$ is the Leray projection operator from $L^2(\Omega)$ onto $H=L^2(\Omega)\cap\{\nabla\cdot u=0\}$, and $A=\mathbb{P}(-\Delta)$ is the Stokes operator, and $B(u,v)=\mathbb{P}(u\cdot\nabla v)$ (see\cite{CF}). It is well known that conditions
\begin{align}
\norm{\nabla u_0}_2&\leq \zeta_1\label{s1}\\
\sup_{t\geq 0}\norm{\mathbb{P}f}_2&\leq \zeta_2\label{s2},
\end{align}
result in the global existence and uniqueness of strong solutions of Navier-Stokes equations, i.e. solutions which belong to $L^{\infty}(0,T; H_0^1(\Omega))\cap L^2(0,T; \mathcal D(A))$. A direct applcation of the these results is not possible because in them the forces $f(t)$ are given. But the proof of the results (see for instance the proof of Theorem 9.3 in \cite{CF}) can be adapted verbatim for our coupled system, as long as we can verify independently conditions (\ref{s2}). 

The fact that (\ref{s2}) is satisfied is due to the remarkable feature of the electrical forcing, which, in steady state is a pure gradient.  Indeed, in view of the definition of the Boltzmann states (\ref{boltz def}), (\ref{boltz2})
 $\rho^*\nabla\Phi^*=-\nabla(\sum_i c_i^*)$ is a pure gradient, and hence vanishes under $\mathbb{P}$. Thus we obtain,
\begin{align*}
\norm{\mathbb{P}(\rho\nabla\Phi)}_2^2&=\norm{\mathbb{P}(\rho\nabla\Phi-\rho^*\nabla\Phi^*)}_2^2\\
																	&\leq 2\norm{\rho\nabla(\Phi-\Phi^*)}_2^2+2\norm{(\rho-\rho^*)\nabla\Phi^*}_2^2\\
																	&\leq 2\norm{\rho}_\infty^2\norm{\nabla(\Phi-\Phi^*)}_2^2+2\norm{\nabla\Phi^*}_\infty^2\norm{\rho-\rho^*}_\infty\norm{\rho-\rho^*}_1\\
																	&\leq \max\left\{2\norm{\rho}_\infty^2\left(\frac{2}{\varepsilon}\right)E_K,\,2C^\frac{1}{2}Nz\norm{\nabla\Phi^*}_\infty^2\norm{\rho-\rho^*}_\infty(\max_i\norm{c_i(0)}_1)^\frac{1}{2}E_K^\frac{1}{2}       \right\}\\
																	&=:\max\{B_1E_K, B_2E_K^\frac{1}{2}\}
\end{align*}
where the last inequality follows from (\ref{diss2}) and (\ref{L1}). Thus we see that for
\begin{align}
E_K\leq\min\left\{\frac{\zeta_2^2}{B_1k_B^2T_K^2},\,\frac{\zeta_2^4}{B_2^2k_B^4T_K^4}\right\}=:\delta_3\label{deltathree}
\end{align}
condition (\ref{s2}) is satisfied. Thus with $\delta_1, \delta_2, \delta_3, \zeta_1$ given by (\ref{deltaone}), (\ref{deltatwo}), (\ref{deltathree}) and (\ref{s1}), the proof of Theorem \ref{theorem} is complete.


\section{\large{Global Existence for Small Perturbations (Uniform Selective)}}
In this section, we consider uniform selective boundary conditions and prove the following result, which extends Theorem \ref{theorem}:
\begin{theorem}\label{theorem2}
Let $\Omega\subset\mathbb{R}^3$ be an open bounded domain with smooth boundary. Let $z_i\in\mathbb{R}$, $i=1,...,N$ (such that there exist $i,j$ with $z_i<0<z_j$), and let $\varepsilon>0$, $D_i>0$, $i=1,...,N$. Let $p=2q>6$, and suppose the following initial data are given: $c_i(0)\geq 0$  in $W^{2,q}(\Omega)$, $i=1,...,N$, $W\in W^{\frac{3}{2},p}(\partial\Omega)$, and $u_0\in (W_0^{1,p}(\Omega))^3\cap\{\nabla\cdot u=0\}$. Furthermore, assume that the following smallness conditions for the initial data are satisfied:
\begin{align*}
E_K&\leq \min\{\delta'_1,\delta'_3\}\\
\sum_{i=1}^N\norm{c_i(0)-c_i^*}_2^2&\leq \delta'_2\\
\norm{\nabla u_0}_2&\leq \zeta_1
\end{align*} 
where the constants $\delta'_1, \delta'_2, \delta'_3, \zeta_1$ depend on the parameters of the problem, boundary data and initial charges, and $c_i^*$ are defined with normalizing constants $Z_i^{-1}=(\gamma_i(x) e^{z_i W(x)})_{|\pa\Omega}$ for $i=1,...,M$, and $Z_i>0$ arbitrary for $i=M+1,...,N$.
Then there exists a unique strong solution of the NPNS system (\ref{eq1})-(\ref{phi bc}) with uniform selective boundary conditions (\ref{unif1})-(\ref{W}) in $\Omega\times[0,\infty)$. 
\end{theorem}

\subsection{Proof of Theorem \ref{theorem2}}
The main goal is the same as for the blocking case - to establish uniform $L^p$ bounds on $c_i$ for $p>6$ and then to apply the local existence theorem.
\subsubsection{Uniform $L^\infty(L^2)$ bounds for $c_i$. Smallness condition. Bounds on $\Phi$.}
\indent The initial estimates in the proof of uniform $L^\infty(L^2)$ bounds for $c_i$ follow with little modification from the blocking boundary condition case. Relation (\ref{L2-1}) is obtained without change (uniform selective boundary conditions with our choice of normalizing constants $Z_i$ ensure that no boundary terms appear from integration by parts). Estimates for $I_1, I_2, I_4$ (i.e. the estimates that do not require interpolation) are unchanged. For $I_3$ and $I_5$, slight technical modifications are required because interpolation inequalities for $c_i-c_i^*$, which is now neither zero trace nor zero mean in general, produce extra terms. Thus for $I_3$ and $I_5$, we have
\begin{align*}
I_3\leq&D_i|z_i|\norm{c_i-c_i^*}_3\norm{\nabla(\Phi-\Phi^*)}_6\norm{\nabla(c_i-c_i)}_2\\
		\leq&C_{3,i}D_i|z_i|\varepsilon^{-1}(d_\Omega^{-\frac{1}{2}}\norm{c_i-c_i^*}_2+\norm{\nabla(c_i-c_i^*)}_2^\frac{1}{2}\norm{c_i-c_i^*}_2^\frac{1}{2})\norm{\nabla(c_i-c_i^*)}_2\norm{\rho-\rho^*}_2\\
		\leq&C_{3,i}D_i\max_j|z_j|^2\varepsilon^{-1}(d_\Omega^{-\frac{1}{2}}\norm{\nabla(c_i-c_i^*)}_2\Sigma_{j=1}^N\norm{c_j-c_j^*}_2^2+\norm{\nabla(c_i-c_i^*)}_2^\frac{3}{2}(\Sigma_{j=1}^N \norm{c_j-c_j^*}_2^2)^\frac{3}{4})\\
		\leq&\epsilon' D_i\norm{\nabla(c_i-c_i^*)}_2^2+C'_{3,i}\frac{D_i\max_j|z_j|^4}{d_\Omega\varepsilon^2}(\Sigma_{j=1}^N\norm{c_j-c_j^*}_2^2)^2+C_{3,i}\frac{D_i\max_j|z_j|^8}{\varepsilon^4}(\Sigma_{j=1}^N\norm{c_j-c_j^*}_2^2)^3\\
I_5\leq&D_i|z_i|\norm{\nabla\Phi^*}_\infty\norm{c_i-c_i^*}_2\norm{\nabla(c_i-c_i^*)}_2\\
		\leq&C_{5,i}D_i|z_i|\norm{\nabla\Phi^*}_\infty(d_\Omega^{-\frac{3}{2}}\norm{c_i-c_i^*}_1\norm{\nabla(c_i-c_i^*)}_2+\norm{c_i-c_i^*}_1^\frac{2}{5}\norm{\nabla(c_i-c_i^*)}_2^\frac{8}{5})\\
		\leq&\epsilon'D_i\norm{\nabla(c_i-c_i^*)}_2^2+C_{5,i}D_i(d_\Omega^{-3}|z_i|^2\norm{\nabla\Phi^*}_\infty^2+|z_i|^5\norm{\nabla\Phi^*}_\infty^5)\norm{c_i-c_i^*}_1^2
\end{align*}
Thus, as before we take $\epsilon'=1/8$ and obtain the following slightly different differential inequality:
\begin{align}
\frac{d}{dt}\sum_{i=1}^N\norm{c_i-c_i^*}_2^2+D^-\sum_{i=1}^N\norm{\nabla(c_i-c_i^*)}_2^2\leq K(t)+G\left(\sum_{i=1}^N\norm{c_i-c_i^*}_2^2\right)^2+H\left(\sum_{i=1}^N\norm{c_i-c_i^*}_2^2\right)^3\label{4-L2}
\end{align}
where
\begin{align}
K(t)&=2N\tilde{C}\left(\beta_1\norm{u}_2^2+\beta_2\norm{\nabla(\Phi-\Phi^*)}_2^2+\beta'_3\max_i \norm{c_i-c_i^*}_1^2\right)\\
G&=2N\tilde{C}_3\frac{D^+z^4}{d_\Omega\varepsilon^2}\\
\tilde{C}_3&=\max_i\{C'_{3,i}\}\\
\beta'_3&=D^+(d_\Omega^{-3}|z_i|^2\norm{\nabla\Phi^*}_\infty^2+z^5\norm{\nabla\Phi^*}_\infty^5)
\end{align}
and all other constants remain unchanged from the blocking case. Then, an application of Poincar\'{e}'s inequality finally gives us
\begin{align}
\frac{dw}{dt}\leq K(t)-C_{\Omega}D^- w+Gw^2+Hw^3
\end{align}
where $w$ is defined as before. The goal is once again to control the size of $K(t)$ with that of $E_K$. However, Lemma \ref{lemma} is not applicable as we no longer have $\norm{c_i}_1=\norm{c_i^*}_1$ in general. At this point it becomes necessary to invoke a slightly more involved result, referred to as the generalized Csiszar-Kullback inequality {\cite{UAMT}}.
\begin{lemma}\label{gck}
Suppose $(\Omega, \Sigma, \mu)$ is a measure space and $\mu$ is a probability measure. Furthermore, assume $g\in L^1(d\mu)$ is strictly positive a.e. and $\norm{g}_1=1$. Then for all $f\in L^1(d\mu)$ such that $f\geq 0$ a.e., we have the inequality
\begin{align}
0\leq(1+\norm{f-g}_1)(\log(1+\norm{f-g}_1)-1)+1\leq\int_\Omega f\log\left(\frac{f}{g}\right)-f+g\,d\mu\label{gckineq}.
\end{align}
\end{lemma}
It is clear that by scaling, upon modifying (\ref{gckineq}) accordingly, we may relax the assumptions that $\mu(\Omega)=1$ and $\norm{g}_1=1$. Then, from Lemma \ref{gck}, taking $f=c_i$ and $g=c_i^*$ we find that while $\norm{c_i-c_i^*}_1$ is not bounded above by $E_K$ as straightforwardly as in (\ref{L1}), it is nonetheless true that for all $\epsilon>0$ there exists $\delta>0$ such that
\begin{align}
E_K\leq \delta\Rightarrow \max_i\sup_{t\geq 0}\norm{c_i-c_i^*}_1\leq \epsilon\label{small limit}
\end{align}
and this is sufficient for our needs. We may then conclude using (\ref{diss2}) and (\ref{small limit}) that for any $\epsilon>0$ there exists $\delta'_1=\delta'_1(\epsilon)$ such that
\begin{align}
E_K\leq \delta'_1\Rightarrow \sup_{t\geq 0}K(t)\leq \epsilon.\label{delta'1}
\end{align}
Then, as before, as long as 
\be
Gw +H w^2 \le \fr{1}{2} C_{\Omega}D^-
\la{smallness}
\ee
we have that
\be
\fr{d w}{dt} \le \epsilon - \lambda w
\la{interneq}
\ee
with $\lambda = \fr{1}{2} C_{\Omega}D^-$ as before, and hence
\be
w(t) \le w(0) + \fr{\epsilon}{\lambda}
\la{internow}
\ee
Then, as before, there exists $\delta_2'>0$ such that by taking 
\be
w(0)\le\delta_2'\la{delta'2}
\ee
and $\epsilon$ small enough, we guarantee that 
the inequality (\ref{internow}) implies (\ref{smallness}) and consequently there  exists $\tilde{w}>0$ such that
\begin{align}
\sup_{t\geq 0}w(t)\leq \tilde{w}.
\end{align}

A consequence of the uniform $L^\infty(L^2)$ bound, which follows from (\ref{4-L2}), is
\begin{align}
\int_{t_0}^{t_0+\tau}\norm{\nabla c_i(t)}_2^2\,dt\leq\chi_2(1+\tau),\,\,\,i=1,...,N\label{nab c}
\end{align}
where $\chi_2$ depends on initial conditions, boundary data, and parameters of the system, but is independent of $t_0$.

\indent Lastly we remark that now we also have at our disposal the bounds (\ref{G1}), (\ref{G2}).

\subsubsection{Local uniform $L^\infty(L^p)$ bounds for $c_i$ for $1\leq p< \infty$}
\indent It is at this step that the analysis largely differs from that in the blocking case. As a preliminary step before showing uniform $L^\infty(L^p)$ bounds, we start by showing \textit{local} uniform bounds (see (\ref{doubling})). We first proceed similarly to the start of subsection \ref{lp bounds}. Integrating (\ref{eq1}) against $c_i^{k-1}-(c^*_i)^{k-1}$ ($k\geq 3$) we obtain
\begin{align}
\frac{1}{k}\frac{d}{dt}\int_\Omega c_i^k-kc_i (c^*_i)^{k-1}\,dx=&-\int_\Omega (-uc_i+D_i\nabla c_i+z_iD_i c_i\nabla\Phi)\cdot\nabla c_i^{k-1}\,dx\nonumber\\
																															&-\int_\Omega(-uc_i+D_i\nabla c_i+z_iD_ic_i\nabla\Phi)\cdot\nabla(c^*_i)^{k-1}\,dx\nonumber\\
																													\leq&-D_i\frac{4(k-1)}{k^2}\int_\Omega|\nabla c_i^\frac{k}{2}|^2\,dx-z_iD_i\frac{2(k-1)}{k}\int_\Omega c_i^\frac{k}{2}\nabla\Phi\cdot\nabla c_i^\frac{k}{2}\,dx\nonumber\\
																													      &+\alpha_k(1+\norm{\nabla c_i}_2^2)\label{lp inequality}
\end{align}
where $\alpha_k$ depends on $k$, bounds on $c_i^*$, uniform $L^\infty(L^2)$ bounds on $u$, $c_i$ and $\nabla\Phi$, and parameters of the system. The second integral in the last inequality is estimated using Hölder inequality, interpolation and (\ref{G1}):
\begin{align}
\int_\Omega c_i^\frac{k}{2}\nabla\Phi\cdot\nabla c_i^\frac{k}{2}\,dx\leq&\norm{c_i^\frac{k}{2}}_3\norm{\nabla\Phi}_6\norm{\nabla c_i^\frac{k}{2}}_2\\
																																 \leq&\epsilon\norm{\nabla c_i^\frac{k}{2}}_2^2+C_k\norm{c_i^\frac{k}{2}}_2^2.
\end{align}
Then, choosing $\epsilon$ appropriately and possibly modifying $\alpha_k$, we obtain from (\ref{lp inequality}) after rearranging:
\begin{align}
\frac{d}{dt}\norm{c_i^\frac{k}{2}}_2^2+\norm{\nabla c_i^\frac{k}{2}}_2^2\leq\alpha_k(\norm{c_i^\frac{k}{2}}_2^2+\norm{\nabla c_i}_2^2+\frac{d}{dt}G(t)+1)\label{k ineq}
\end{align}
where we have defined
\begin{align}
G(t):=\int_\Omega c_i(t)(c^*_i)^{k-1}\,dx.
\end{align}
Then, setting $X(t):=\norm{c_i^\frac{k}{2}}_2^2$ and using (\ref{nab c}), we use a Gr\"{o}nwall argument with integrating factor $e^{-\alpha_k t}$ and restricting ourselves to the time range $(t_0,\, t_0+\tau)$:
\begin{align}
&\frac{d}{dt}(X(t)e^{-\alpha_k t})=e^{-\alpha_k t}(X'(t)-\alpha_k X(t))\leq\alpha_k e^{-\alpha_k t_0}(\norm{\nabla c_i}_2^2+\frac{d}{dt}G(t)+1)\nonumber\\
\Rightarrow&X(t_0+\tau)e^{-\alpha_k(t_0+\tau)}\leq X(t_0)e^{-\alpha_k t_0}+\alpha_k e^{-\alpha_k t_0}(2\chi_2(1+\tau)+G(t_0+\tau)-G(t_0))\nonumber\\
\Rightarrow&X(t_0+\tau)\leq A_k(X(t_0)+1)e^{(1+\alpha_k)\tau}
\end{align}
where in the last implication we used the fact that $G(t)$ is bounded independent of time and the inequality $\tau\leq e^\tau$, and $A_k$ depends on bounds on $G$, $\alpha_k$ and $\chi_2$. Thus we have shown that for each $k\geq 3$, there exists a constant $B_k$ independent of $t_0$ such that
\begin{align}
\sup_{t_0\leq t\leq t_0+\tau}\norm{c_i(t)}_k\leq e^{B_k(1+\tau)}(1+\norm{c_i(t_0)}_k).\label{doubling}
\end{align}

\subsubsection{Local uniform $L^1(L^6)$ bounds for $c_i$}
The local uniform bound (\ref{nab c}) together with the embedding $H^1(\Omega)\hookrightarrow L^6(\Omega)$ gives us
\begin{align}
\int_{t_0}^{t_0+\tau}\norm{c_i}_6\,dt\leq\eta_6(1+\tau)
\label{local l2l3}
\end{align}
with $\eta_6$ not depending on $t_0$.

\subsubsection{Uniform $L^\infty(L^p)$ bounds for $c_i$ for $p\leq 18$.} 
It turns out that with the following uniform Gr\"{o}nwall lemma, the bounds (\ref{doubling}) and (\ref{local l2l3}) give us  uniform $L^\infty(L^6)$ bounds (this lemma is also used extensively in {\cite{PC}} for essentially the same purpose).
\begin{lemma}(Uniform Gr\"{o}nwall Lemma) Let $r:\mathbb{R}^+\to\mathbb{R}^+$ be nondecreasing, and suppose that $0\leq f(t)\in L^1_{loc}([0,T])$ satisfies, for all $[t_0,t_0+\tau]\subset[0,T]$,
\begin{align}
\sup_{t_0\leq t\leq t_0+\tau}f(t)&\leq r(f(t_0))e^{c_1(1+\tau)}\label{B1}\\
\int_{t_0}^{t_0+\tau}f(t)\,dt&\leq c_2(1+\tau)\label{B2}
\end{align}
with constants $c_1, c_2>0$ independent of $t_0$. Then there exists $c_\tau>0$, depending on $\tau$ and $f(0)$, such that
\begin{align}
\sup_{0\leq t\leq T}f(t)\leq c_\tau.\label{unif}
\end{align}\label{unifgronwall}
\end{lemma}
\begin{remark}
In the lemma above, $\tau$ is a \textit{fixed} time increment and the dependence of $c_\tau$ on $\tau$ should not be mistaken with dependence on time (i.e. $T$). That is, (\ref{unif}) is indeed a uniform in time bound.
\end{remark}
\begin{proof}
(adapted from \cite{PC}) Taking $t_0=0$ and replacing $\tau$ with $\tau/2$ in (\ref{B2}), Chebyshev's inequality tells us
\begin{align}
\mu(\{t\in [0,\tau/2]: f(t)\geq 4c_2(1+\tau)/\tau\})\leq \tau/4
\end{align}
where $\mu$ is the Lebesgue measure on $\mathbb{R}$. Since $\{t\in [0,\tau/2]: f(t)\geq 4c_2(1+\tau)/\tau\}$ has strictly less than full measure on $[0,\tau/2]$, there exists $t_0\in[0,\tau/2]$ such that
\begin{align}
f(t_0)\leq \frac{4c_2(1+\tau)}{\tau}.\label{tau}
\end{align}
With this value of $t_0$, the local uniform bound (\ref{B1}) tells us in particular that
\begin{align}
\sup_{\frac{\tau}{2}\leq t\leq \tau}f(t)\leq r\left(\frac{4c_2(1+\tau)}{\tau}\right)e^{c_1(1+\tau)}.\label{tau2}
\end{align}
We can repeat the preceeding procedure on the interval $[\tau/2,\tau]$: there exists $t_0\in [\tau/2,\tau]$ such that (\ref{tau}) holds with the same bound. Then $(\ref{tau2})$ holds on $[\tau,3\tau/2]$ with the same bound. Thus, inductively, by adjoining time intervals of length $\tau/2$, we obtain
\begin{align}
\sup_{\frac{\tau}{2}\leq t\leq T}f(t)\leq r\left(\frac{4c_2(1+\tau)}{\tau}\right)e^{c_1(1+\tau)}.\label{tau3}
\end{align}
where the right hand side does not depend on $T$. Finally, by adding estimate (\ref{B1}) with $t_0=0$ we obtain
\begin{align}
\sup_{0\leq t\leq T}f(t)\leq c_\tau.\label{tau4}
\end{align}
with
\begin{align}
c_\tau=\left(r(f(0))+r\left(\frac{4c_2(1+\tau)}{\tau}\right)\right)e^{c_1(1+\tau)}
\end{align}
\end{proof}
Applying the lemma to (\ref{doubling}) (with $k=6$) and (\ref{local l2l3}), we obtain
\begin{align}
\sup_{t\geq 0}\norm{c_i}_6\leq \sigma_6
\end{align}
with $\sigma_6$ independent of time. We have the following consequence: returning to (\ref{k ineq}), we see by taking $k=6$ that
\begin{align}
\int_{t_0}^{t_0+\tau}\norm{\nabla c_i^3}_2^2\,dt\leq \eta'_6(1+\tau)
\end{align}
and by the embedding $H^1(\Omega)\hookrightarrow L^6(\Omega)$, we have for $p\leq 18$
\begin{align}
\int_{t_0}^{t_0+\tau}\norm{c_i}_p\,dt\leq\eta_p(1+\tau)\label{etap}
\end{align}
with constants $\eta_p$ independent of $t_0$.
\begin{remark}
We actually get local uniform $L^6(L^p)$ bounds from the embedding, but we will just be needing the weaker bound (\ref{etap}).
\end{remark}
Then, with (\ref{etap}) and (\ref{doubling}), we again apply Lemma \ref{unifgronwall} to obtain for $p\leq 18$,
\begin{align}
\sup_{t\geq 0}\norm{c_i}_p\leq \sigma_p\label{sigmap}
\end{align}
with $\sigma_p$ independent of time.

\subsubsection{Uniform $L^\infty(L^p)$ bounds for $c_i$ for $1\leq p <\infty$}
With $(\ref{sigmap})$, we can return to $(\ref{k ineq})$ and obtain bounds of the form
\begin{align}
\int_{t_0}^{t_0+\tau}\norm{\nabla c_i^\frac{p}{2}}_2^2\,dt\leq \eta'_p(1+\tau)
\end{align}
for $p\leq 18$. Then Sobolev embedding gives us bounds of the form (\ref{etap}) for larger $p$ (specifically for $p$ up to $p=54$). Thus together with (\ref{doubling}), we can repeat the above process indefinitely, obtaining bounds (\ref{etap}) and (\ref{sigmap}) for successively larger $p$. Ultimately we obtain for all $p<\infty$, 
\begin{align}
\sup_{t\geq 0}\norm{c_i}_p\leq \sigma_p\label{sigmapp}
\end{align}
with $\sigma_p$ independent of time.

\subsubsection{Global regularity for Navier-Stokes subsystem}\label{NSE bound2}
As in the blocking case, we would like to show that $\norm{\mathbb{P}(\rho\nabla\Phi)}_2$ can made made small uniformly in time, given small initial conditions. Proceeding slightly differently from before, we have from interpolation and elliptic regualrity,
\begin{align}
\norm{\mathbb{P}(\rho\nabla\Phi)}_2=&\norm{\mathbb{P}(\rho\nabla\Phi-\rho^*\nabla\Phi^*)}_2\\
																	\leq&\norm{\rho\nabla(\Phi-\Phi^*)}_2+\norm{(\rho-\rho^*)\nabla\Phi^*}_2\\
																	\leq&C(\norm{\rho}_3\norm{\nabla(\Phi-\Phi^*)}_6+\norm{\rho-\rho^*}_2)\\
																	\leq&C\norm{\rho-\rho^*}_2
\end{align}
where the final $C$ depends on uniform bounds on $\rho$, bounds on $\nabla\Phi^*$, and $\varepsilon$. Then, in order to conclude, it suffices to note that there exists $\delta'_3>0$ such that
\begin{align}
E_K\leq \delta'_3\Rightarrow k_BT_K\norm{\mathbb{P}(\rho\nabla\Phi)}_2\leq\zeta_2\label{delta'3}
\end{align}
where $\zeta_2$ is from (\ref{s2}). Thus with smallness constants $\delta'_1, \delta'_2, \delta'_3, \zeta_1$ defined by (\ref{delta'1}), (\ref{delta'2}), (\ref{delta'3}), (\ref{s1}), the proof of Theorem \ref{theorem2} is complete.

\begin{section}{\large{Conclusion}}
We have shown global existence of strong solutions to the three dimensional Nernst-Planck-Navier-Stokes system in a bounded domain with initial data that are sufficiently small perturbations of steady state solutions (Boltzmann state and zero fluid velocity). The result is shown for both blocking (no-flux) and uniform selective (special Dirichlet) boundary conditions for ionic concentrations.The solutions remain for all time close to the equilibrium solutions in strong norms. The main two steps of the proof are (1) the decay of the sum of
relative entropies (Kullback-Leibler divergences) and (2) the control of $L^2$ 
norms of deviations by the sum of relative entropies.

\end{section}

\vspace{.5cm}

{\bf{Acknowledgment.}} The work of PC was partially supported by NSF grant DMS-
1713985.


\end{document}